\documentclass[english]{article}
\usepackage[T1]{fontenc}
\usepackage[latin9]{inputenc}
\usepackage{babel}
\usepackage{amsthm}
\usepackage{amsmath}
\usepackage{amssymb}
\usepackage{esint}
\usepackage[unicode=true,pdfusetitle,
 bookmarks=true,bookmarksnumbered=false,bookmarksopen=false,
 breaklinks=false,pdfborder={0 0 1},backref=false,colorlinks=false]
 {hyperref}

\makeatletter

\newcommand{\noun}[1]{\textsc{#1}}

\theoremstyle{plain}
\newtheorem{thm}{\protect\theoremname}
  \theoremstyle{definition}
  \newtheorem{defn}[thm]{\protect\definitionname}
  \theoremstyle{remark}
  \newtheorem{rem}[thm]{\protect\remarkname}
  \theoremstyle{plain}
  \newtheorem{lem}[thm]{\protect\lemmaname}
  \theoremstyle{plain}
  \newtheorem{cor}[thm]{\protect\corollaryname}

\def\LyX{\texorpdfstring{%
  L\kern-.1667em\lower.25em\hbox{Y}\kern-.125emX\@}
  {LyX}}

\makeatother

\usepackage{babel}

\makeatother

\usepackage{babel}

\makeatother

\usepackage{babel}

\date{}

\makeatother

  \providecommand{\corollaryname}{Corollary}
  \providecommand{\definitionname}{Definition}
  \providecommand{\lemmaname}{Lemma}
  \providecommand{\remarkname}{Remark}
\providecommand{\theoremname}{Theorem}

\begin{document}

\title{\textbf{\noun{On Hofer Energy of $J$-holomorphic Curves for Asymptotically
Cylindrical $J$}}}

\author{Erkao Bao}
\maketitle
\begin{abstract}
In this paper, we provide a bound for the generalized Hofer energy
of punctured $J$-holomorphic curves in almost complex manifolds with
asymptotically cylindrical ends. As an application, we prove a version
of Gromov's Monotonicity Theorem with multiplicity. Namely, for a
closed symplectic manifold $(M,\omega')$ with a compatible almost
complex structure $J$ and a ball $B$ in $M,$ there exists a constant
$\hbar>0,$ such that any $J$-holomorphic curve $\tilde{u}$ passing
through the center of $B$ for $k$ times (counted with multiplicity)
with boundary mapped to $\partial B$ has symplectic area $\int_{\tilde{u}^{-1}(B)}\tilde{u}^{*}\omega'>k\hbar,$
where the constant $\hbar$ depends only on $(M,\omega',J)$ and the
radius of $B.$ As a consequence, the number of times that any closed
$J$-holomorphic curve in $M$ passes through a point is bounded by
a constant depending only on $(M,\omega',J)$%
\footnote{Following the notation in \cite{compactness} we save $\omega$ for
something else.%
} and the symplectic area of $\tilde{u}$. Here $J$ is any $\omega'-$compatible
smooth almost complex structure on $M$. In particular, we do not
require $J$ to be integrable.\end{abstract}
\begin{verse}
\end{verse}

\section{Introduction}

Hofer energy is introduced in \cite{Hofer Weinstein conjecture} for
$J$-holomorphic curves in symplectization of contact manifolds, and
is generalized in \cite{compactness} for $J-$holomorphic curves
in the ``almost complex manifolds with cylindrical ends''. Here
``cylindrical'' means that the almost complex structure $J$ is
invariant under translation. Hofer energy plays an essential role
in the study of $J$-holomorphic curves in Symplectic Field Theory
mainly because of the following two properties: (A) the asymptotic
behavior of a $J$-holomorphic curve in a noncompact symplectic manifold
can be controlled by requiring its Hofer energy to be finite, and
hence a uniform Hofer energy bound gives a Symplectic Field Theory
type of compactification of moduli spaces of $J-$holomorphic curves;
on the other hand, (B) a uniform Hofer energy bound can be obtained
by specifying the behavior the $J$-holomorphic curves at infinity
and bounding their symplectic areas (see \cite{Hofer Weinstein conjecture,compactness}).
In \cite{asympt} the notion of Hofer energy and Property (A) are
further generalized to include $J$-holomorphic curves in ``almost
complex manifolds with asymptotically cylindrical ends''. Here ``asymptotically
cylindrical'' means that the difference between the almost complex
structure $J$ and a translation invariant one is exponentially small.
In this paper, we prove Property (B) in this setting. Property (A)
and property (B) together imply the expected useful compactness results
in Symplectic Field Theory. 

One of the main advantages of this generalization is that the asymptotically
cylindrical $J$ arises naturally. As an application, we prove a version
of Gromov's Monotonicity Theorem with multiplicity%
\footnote{This can also be derived from \cite{fish}. See Remark \ref{rmk: Fish}%
}, namely for a closed symplectic manifold $(M,\omega')$ with a compatible
almost complex structure $J$ and a ball $B$ in $M,$ there exists
a constant $\hbar>0,$ such that any $J$-holomorphic curve $\tilde{u}$
passing through the center of $B$ $k$ times (counted with multiplicity)
with the boundary mapped to $\partial B$ has symplectic area $\int_{\tilde{u}^{-1}(B)}\tilde{u}^{*}\omega'>k\hbar,$
where the constant $\hbar$ depends only on $(M,\omega',J)$ and the
radius of $B.$ 

The inequality $k<\frac{1}{\hbar}\int_{\tilde{u}^{-1}(B)}\tilde{u}^{*}\omega'$
is closely related to a question asked in \cite{cieliebak}, where
they study $J$-holomorphic curves with boundaries lying inside two
clean intersecting Lagrangian submanifolds, and prove that the number
of ``boundary switches'' at the intersecting loci is uniformly bounded
by the Hofer Energy. Their proof in an essential way relies on the
additional requirement that the almost complex structure $J$ is integrable
near the intersecting loci. They ask to what extent their results
are still true without assuming the integrability of $J.$ In this
paper, we provide a simple proof for the closed version of their result
for arbitrary $J$. Namely, the $J$-holomorphic curves we consider
in this paper have no boundaries. In this case, ``boundary switches''
just means that the $J$-holomorphic curve passes a fixed point in
$M.$ Furthermore, the analysis developed in \cite{asympt} and this
paper can be carried out to include Lagrangians without difficulty
(see for example section 5 in \cite{asympt} for the setup).

\subsection*{Acknowledgment}

I would like to thank Garrett Alston, Xianghong Gong, Conan Leung,
Yong-Geun Oh, and Dietmar A. Salamon for helpful discussions. I would
like to thank the anonymous referee for the instructive and critical
comments. In particular, they suggest me to include Theorem \ref{thm: hofer's energy bd by symplectic area-1}.
I would thank Conan Leung for providing me such a great opportunity
to visit Institute of Mathematical Sciences at the Chinese University
of Hong Kong, where I completed this paper.

\section{\label{sec:Asymptotically...}Asymptotically cylindrical almost complex
structure}

Let $V_{-}$ be a smooth closed oriented manifold of dimension $2N-1$,
and $J$ be a smooth almost complex structure on $W_{-}=\mathbb{R}^{-}\times V_{-}$
such that the orientation of $W_{-}$ induced from $J$ conincides
with the one induced from the standard of orientation of $\mathbb{R}^{-}$
and the orientation of $V_{-}.$ Let $\mathbf{R}$ be the smooth vector
field on $W_{-}$ defined by $\mathbf{R}:=J\left(\frac{\partial}{\partial r}\right),$
and $\xi$ be the subbundle of the tangent bundle $TW_{-}$ defined
by $\xi_{(r,v)}=\left(JT_{v}\left(\{r\}\times V_{-}\right)\right)\cap\left(T_{v}\left(\{r\}\times V_{-}\right)\right)$,
for $(r,v)\in W_{-}$. Then the tangent bundle $TW_{-}$ splits as
$TW_{-}=\mathbb{R}(\frac{\partial}{\partial r})\oplus\mathbb{R}(\mathbf{R})\oplus\xi$.
Define the 1-forms $\lambda$ and $\sigma$ on $W_{-}$ respectively
by 
\begin{equation}
\begin{array}{ccccc}
\lambda(\xi)=0 &  & \lambda(\frac{\partial}{\partial r})=0 &  & \lambda\left(\mathbf{R}\right)=1,\end{array}\label{eq:define lambda}
\end{equation}
\begin{equation}
\begin{array}{ccccc}
\sigma(\xi)=0 &  & \sigma(\frac{\partial}{\partial r})=1 &  & \sigma\left(\mathbf{R}\right)=0.\end{array}\label{eq:define sigma}
\end{equation}

Let $f_{s}:W_{-}\to W_{-}$ be the translation $f_{s}(r,v)=(r+s,v),$
for $s\leqq0.$ We call a tensor on $W_{-}$ translationally invariant
if it is invariant under $f_{s}$.
\begin{defn}
\label{def: asympt cylindrical}Under the above notations, $J$ is
called asymptotically cylindrical at negative infinity, if $J$ satisfies
(ACC1)-(ACC5):\end{defn}
\begin{itemize}
\item (ACC1) There exist a smooth translationally invariant almost complex
structure $J_{-\infty}$ on $W_{-}$ and constants $K_{l},\delta_{l}>0$,
such that restricted to the region $(-\infty,r]\times V_{-}$ 
\begin{equation}
\left\Vert J-J_{-\infty}\right\Vert _{l}\leqq K_{l}e^{\delta_{l}r}
\end{equation}
 for all $r\leqq0$ and $l\in\mathbb{Z}_{\geqq0}$, where $\left\Vert \cdot\right\Vert _{k}$
is the $\textrm{C}^{k}$-norm defined by $\left\Vert \varphi\right\Vert _{k}:=\underset{w}{\sup}\sum_{i=0}^{k}\left|\nabla^{i}\varphi(w)\right|$
and $|\cdot|$ is computed using a translationally invariant metric
$g_{W_{-}}$ on $W_{-}$, for example $g_{W_{-}}=dr^{2}+g_{V_{-}},$
and $\nabla$ is the corresponding Levi-Civita connection. 
\item (ACC2) $i(\mathbf{R}_{-\infty})d\lambda_{-\infty}=0,$ where $\mathbf{R}_{-\infty}:=\underset{s\to-\infty}{\lim}f_{s}^{*}\mathbf{R}$,
$\lambda_{-\infty}:=\underset{s\to-\infty}{\lim}f_{s}^{*}\lambda$,
and both limits exist by (ACC1).
\item (ACC3) $\mathbf{R}_{-\infty}(r,v)\in T_{v}(\{r\}\times V_{-})$, i.e.
$\mathbf{R}_{-\infty}$ is tangent to the level sets. 
\end{itemize}
There exists a translationally invariant closed 2-form $\omega_{-\infty}$
on $W_{-}$ such that 
\begin{itemize}
\item (ACC4) $i\left(\frac{\partial}{\partial r}\right)\omega_{-\infty}=0=i(\mathbf{R}_{-\infty})\omega_{-\infty}.$
\item (ACC5) $\omega_{-\infty}|_{\xi_{-\infty}}(\cdot,J_{-\infty}\cdot)$
is a metric on $\xi_{-\infty}:=\underset{s\to-\infty}{\lim}f_{s}^{*}\xi.$ 
\end{itemize}
When we say $J$ is asymptotically cylindrical, we choose $\omega_{-\infty}$
without mentioning. 

Similarly, we could define the notion of $J$ being asymptotically
cylindrical at positive infinity for $W_{+}=\mathbb{R}^{+}\times V_{+}$. 

Notice that this definition is equivalent to the definition given
in \cite{asympt}. In \cite{asympt} for $J$ being asymptotically
cylindrical, besides (ACC1)-(ACC5) we require that there exists a
2-form $\omega$ on $W_{-}$ such that
\begin{itemize}
\item (a) $i\left(\frac{\partial}{\partial r}\right)\omega=0=i(\mathbf{R})\omega.$
\item (b) $\omega|_{\xi}(\cdot,J\cdot)$ is a metric on $\xi.$ 
\item (c) There exist constants $K_{l},\delta_{l}\geqq0$, such that
\begin{equation}
\left\Vert \left.\left(\omega-\omega_{-\infty}\right)\right|_{(-\infty,r]\times V_{-}}\right\Vert _{l}\leqq K_{l}e^{\delta_{l}r}\label{eq:omega-omega -infty is small}
\end{equation}
 for all $r\leqq0$ and $l\in\mathbb{Z}_{\geqq0}$.
\end{itemize}
Indeed, take 
\[
\omega(x,y)=\frac{1}{2}\left[\omega_{-\infty}(\pi_{\xi}x,\pi_{\xi}y)+\omega_{-\infty}(J\pi_{\xi}x,J\pi_{\xi}y)\right]
\]
 for $x,y\in T_{(r,v)}W^{-}.$ Then (a) is satisfied. From (ACC1)
and (ACC4) we can see that (c) is satisfied. Notice 

\begin{eqnarray*}
\omega(Jx,Jy) & = & \frac{1}{2}\left[\omega_{-\infty}(\pi_{\xi}Jx,\pi_{\xi}Jy)+\omega_{-\infty}(J\pi_{\xi}Jx,J\pi_{\xi}Jy)\right]\\
 & = & \frac{1}{2}\left[\omega_{-\infty}(J\pi_{\xi}x,J\pi_{\xi}y)+\omega_{-\infty}(\pi_{\xi}x,\pi_{\xi}y)\right]\\
 & = & \omega(x,y).
\end{eqnarray*}
Hence $\omega|_{\xi}(\cdot,J\cdot)$ is symmetric. For $x\in\xi_{(r,v)},$
we have 
\begin{eqnarray*}
\omega(x,Jx) & = & \frac{1}{2}\left[\omega_{-\infty}(\pi_{\xi}x,\pi_{\xi}Jx)+\omega_{-\infty}(J\pi_{\xi}x,J\pi_{\xi}Jx)\right]\\
 & = & \frac{1}{2}\left[\omega_{-\infty}(x,Jx)+\omega_{-\infty}(Jx,-x)\right]\\
 & = & \omega_{-\infty}(x,Jx).
\end{eqnarray*}
Because that $\omega_{-\infty}(x,J_{-\infty}x)$ is positive on every
nonzero vector $x\in\xi_{-\infty},$ we have $\omega(\cdot,J_{-\infty}\cdot)|_{S}>\varpi>0,$
for some $\varpi,$ where 

\[
S:=\left\{ \left.(x,y)\in\xi_{-\infty}\times\xi_{-\infty}\right|\left\Vert x\right\Vert _{g_{W_{-}}}=1,y=J_{-\infty}x\right\} .
\]
 When $r$ is sufficiently negative, by (ACC1), $(x,Jx)$ is uniformly
close to $S,$ for all $x\in\xi_{(r,v)}.$ Therefore, for $0\neq x\in\xi_{(r,v)}$,
we obtain $\omega(x,Jx)=\omega_{-\infty}(x,Jx)>0,$ and hence (b).
Since we restrict ourselves to the behaviors of $J$-holomorphic curves
near infinity, for the purpose of simplifying the notations, we assume
$\omega$ satisfies (b) for $r\leqq0.$
\begin{rem}
(ACC1)-(ACC5) imply that $(V_{-},\omega_{-\infty})$ is a stable hamiltonian
structure and $(\lambda_{-\infty},J_{-\infty})$ is a framing of the
stable hamiltonian structure (See \cite{stable hamiltonian} for the
definition of stable hamiltonian structure. In this paper we do not
need it).\end{rem}
\begin{defn}
We say an asymptotically cylindrical almost complex structure $J$
is of contact type if $\omega_{-\infty}=d\lambda_{-\infty}.$
\end{defn}

The following definition is the case considered in \cite{Hofer Weinstein conjecture,Finite energy plane,morse bott,compactness}.
\begin{defn}
\label{def:cylindrical almost complex}We say $J$ is a cylindrical
almost complex structure, if $J$ is an asymptotically cylindrical
almost complex structure and translationally invariant. 
\end{defn}
By (ACC2) and (ACC3) we can see that $\mathbf{R}_{-\infty}$ is a
translationally invariant vector field on $W_{-}$ and it is tangent
to each level set $\{r\}\times V_{-}$, so we can view $\mathbf{R}_{-\infty}$
as a vector field on $V_{-}$. Let $\phi^{t}$ be the flow of $\mathbf{R}_{-\infty}$
on $V_{-}$, i.e. $\phi^{t}:V_{-}\to V_{-}$ satisfies $\frac{d}{dt}\phi^{t}=\mathbf{R}_{-\infty}\circ\phi^{t}$.
Then we have 
\[
\frac{d}{dt}[(\phi^{t})^{*}\lambda_{-\infty}]=(\phi^{t})^{*}(i_{\mathbf{R}_{-\infty}}d\lambda_{-\infty}+di_{\mathbf{R}_{-\infty}}\lambda_{-\infty})=0.
\]
Thus $\phi^{t}$ preserves $\lambda_{-\infty}$ and hence $\xi_{-\infty}$.
Similarly $\phi^{t}$ preserves $\omega_{-\infty}$. 

Let's denote by $\mathcal{P}_{-}$ the set of periodic trajectories,
counting their multiples, of the vector field $\mathbf{R}_{-\infty}$
restricting to $V_{-}.$ Notice that any smooth family of periodic
trajectories from $\mathcal{P}_{-}$ have the same period by Stokes'
Theorem and (ACC2).

\begin{defn}
\label{def:Morse Bott}We say that an asymptotically cylindrical $J$
is Morse-Bott if, for every $T>0$ the subset $N_{T}\subseteq V_{-}$
formed by the closed trajectories from $\mathcal{P}_{-}$ of period
$T$ is a smooth closed submanifold of $V_{-}$, such that the rank
of $\omega_{-\infty}|_{N_{T}}$ is locally constant and $T_{p}N_{T}=\ker\left(d\phi^{T}-Id\right)_{p}$.
\end{defn}
In this paper, we assume that $J$ is Morse-Bott. The Morse-Bott condition
is the condition assumed in \cite{asympt} to guarantee Theorem \ref{thm:converge to reeb orbit},
Lemma \ref{lem: hofer energy bound for single curve} and Theorem
\ref{thm:compactification}. For the application in section \ref{sec:An-application-to},
it is easy to check that this requirement is satisfied. \\

Let $\Sigma:=\mathbb{R}^{-}\times S^{1}$ be the half cylinder with
standard almost complex structure $j$, and $\tilde{u}=(a,u):(\Sigma,j)\to(W_{-},J)$
be a $J$-holomorphic curve, i.e. $T\tilde{u}\circ j=J(\tilde{u})\circ T\tilde{u}$.
The $\omega$-energy and $\lambda$-energy of $\tilde{u}$ are defined
as follows respectively

\[
E_{\omega}(\tilde{u})=\int_{\Sigma}\tilde{u}^{*}\omega,
\]
\[
E_{\lambda}(\tilde{u})=\underset{\phi\in\mathcal{C}}{sup}\int_{\Sigma}\tilde{u}^{*}(\phi(r)\sigma\wedge\lambda),
\]
where $\mathcal{C}=\{\phi\in C^{\infty}(\mathbb{R}^{-},[0,1])|\int_{-\infty}^{0}\phi(x)dx=1\}$,
and $\lambda$ and $\sigma$ are defined as in (\ref{eq:define lambda})
and (\ref{eq:define sigma}). The Hofer energy of $\tilde{u}$ is
defined by 
\[
E(\tilde{u})=E_{\omega}(\tilde{u})+E_{\lambda}(\tilde{u}).
\]

Let's equip $\mathbb{R}^{-}\times S^{1}$ with coordinate $(s,t).$
Here we view $S^{1}$ as $\mathbb{R}/\mathbb{Z}$. It is easy to check
that $\tilde{u}^{*}\omega$ and $\tilde{u}^{*}(\phi(r)\sigma\wedge\lambda)$
are non-negative multiples of the volume form $ds\wedge dt$ on $\mathbb{R}^{-}\times S^{1}.$
Actually, 

\begin{equation}
\tilde{u}^{*}\omega=\omega(\pi_{\xi}\tilde{u}_{s},J(\tilde{u})\pi_{\xi}\tilde{u}_{s})ds\wedge dt,\label{eq:positive omega}
\end{equation}
where $\pi_{\xi}$ is the projection from $TW_{-}=\mathbb{R}(\frac{\partial}{\partial r})\oplus\mathbb{R}(\mathbf{R})\oplus\xi$
to $\xi,$ and

\begin{equation}
\tilde{u}^{*}(\phi(r)\sigma\wedge\lambda)=\phi(a)\left[\sigma(\tilde{u}_{s})^{2}+\lambda(\tilde{u}_{s})^{2}\right]ds\wedge dt.\label{eq:positive sigma wedge lambda}
\end{equation}

The non-negativity is the main reason that we choose the Hofer energy
in this form. 

The following theorem is one of the most important theorems in \cite{Hofer Weinstein conjecture,Finite energy plane,compactness,morse bott}
for the case when $J$ is cylindrical, and it is proved in the asymptotically
cylindrical setting in \cite{asympt}.
\begin{thm}
\label{thm:converge to reeb orbit} Suppose that $J$ is an asymptotically
cylindrical almost complex structure on $W_{-}=\mathbb{R}^{-}\times V_{-}$.
Let $\tilde{u}=(a,u):\mathbb{R}^{-}\times\mathbb{R}/\mathbb{Z}\to W_{-}$
be a $J$-holomorphic curve with finite Hofer energy. Suppose that
the image of $\tilde{u}$ is unbounded in $W_{-}$. Then there exists
a periodic orbit $\gamma$ of $\mathbf{R}_{-\infty}$ of period $|T|$
with $T\neq0$, such that

\begin{equation}
\underset{s\to-\infty}{\lim}u(s,t)=\gamma(Tt)\label{eq:u to gamma}
\end{equation}
\begin{equation}
\underset{s\to-\infty}{\lim}\frac{a(s,t)}{s}=T\label{eq:a to Ts}
\end{equation}
 in $C^{\infty}(S^{1})$. 
\end{thm}
On the other hand, we have 
\begin{lem}
\label{lem: hofer energy bound for single curve} Suppose that $J$
is an asymptotically cylindrical almost complex structure on $W_{-}=\mathbb{R}^{-}\times V_{-},$
and $\tilde{u}=(a,u):\mathbb{R}^{-}\times\mathbb{R}/\mathbb{Z}\to W_{-}$
is a $J$-holomorphic curve. Suppose that there exits a periodic orbit
$\gamma$ of $\mathbf{R}_{-\infty}$ of period $|T|$ such that 

\[
\underset{s\to-\infty}{\lim}a(s,t)=-\infty,
\]
\[
\underset{s\to-\infty}{\lim}u(s,t)=\gamma(Tt).
\]
Then 
\[
\underset{s\to-\infty}{\lim}\frac{a(s,t)}{s}=T,
\]
and Hofer energy $E(\tilde{u})<\infty.$ \end{lem}
\begin{proof}
This follows immediately from the proof of Theorem 2 in \cite{asympt}.
Namely, from the assumption, we could derive that the convergence
in (\ref{eq:u to gamma}) and (\ref{eq:a to Ts}) is exponentially
fast. Then it follows by definition and direct calculation that $E(\tilde{u})<\infty$.\end{proof}
\begin{rem}
Theorem \ref{thm:converge to reeb orbit} and Lemma \ref{lem: hofer energy bound for single curve}
also hold for $W_{+}$.
\end{rem}

\section{\label{sec:Almost-complex-manifolds with ends}Almost complex manifolds
with asymptotically cylindrical ends}

Now we introduce the notion of almost complex manifolds with asymptotically
cylindrical ends. 

Let $(E,J)$ be a $2N$ dimensional noncompact almost complex manifold,
and $W_{\pm}$ be an open subset containing the positive (negative)
end of $E$. Assume that $W_{\pm}$ is diffeomorphic to $\mathbb{R}^{\pm}\times V_{\pm}$,
where $V_{\pm}$ is a $2N-1$ dimensional closed manifold. Assume
that there exists a $J$-compatible symplectic form $\omega'$ on
$E,$ and that $J|_{W_{\pm}}$ is an asymptotically cylindrical almost
complex structure at positive (negative) infinity, then we say $(E,J)$
is an almost complex manifold with asymptotically cylindrical positive
(negative) ends.

Let $\tilde{u}$ be a $J$-holomorphic map from a possibly punctured
Riemann surface $(\Sigma,j)$ to $(E,J)$, and then we define for
$a\geqq0,$

\[
E_{symp,a}(\tilde{u})=\int_{\tilde{u}^{-1}\left(E\backslash W_{+}^{a}\bigcup W_{-}^{a}\right)}\tilde{u}^{*}\omega',
\]
 where $W_{+}^{a}:=(a,+\infty)\times V_{+}\subset W_{+},$ and $W_{-}^{a}:=(-\infty,-a)\times V_{-}\subset W_{-}.$

\[
E_{\omega}(\tilde{u})=\int_{\tilde{u}^{-1}(W_{+})}\tilde{u}^{*}\omega+\int_{\tilde{u}^{-1}(W_{-})}\tilde{u}^{*}\omega,
\]

\[
E_{\lambda}(\tilde{u})=\underset{\phi\in\mathcal{C}_{+}}{\sup}\int_{\tilde{u}^{-1}(W_{+})}\tilde{u}^{*}(\phi(r)\sigma\wedge\lambda)+\underset{\phi\in\mathcal{C}_{-}}{\sup}\int_{w^{-1}(W_{-})}\tilde{u}^{*}(\phi(r)\sigma\wedge\lambda),
\]
 where 
\[
\mathcal{C}_{+}=\left\{ \phi\in C^{\infty}(\mathbb{R}^{+},[0,1])|\int\phi=1\right\} ,
\]
 
\[
\mathcal{C}_{-}=\left\{ \phi\in C^{\infty}(\mathbb{R}^{-},[0,1])|\int\phi=1\right\} ,
\]
 and 
\[
E_{a}(\tilde{u})=E_{symp,a}(\tilde{u})+E_{\omega}(\tilde{u})+E_{\lambda}(\tilde{u}).
\]

If $\underset{a\to+\infty}{\lim}E_{symp,a}(\tilde{u})$ is finite,
we define 
\[
E_{symp}(\tilde{u})=\underset{a\to+\infty}{\lim}E_{symp,a}(\tilde{u})
\]
 and 
\[
E(\tilde{u})=E_{symp}(\tilde{u})+E_{\omega}(\tilde{u})+E_{\lambda}(\tilde{u}).
\]

To compactify the moduli space of $J$-holomorphic curves, we need
to include holomorphic buildings (see \cite{compactness}). There
is no difference between almost complex manifolds with cylindrical
ends and almost complex manifolds with asymptotically cylindrical
ends when it comes to the definition of holomorphic buildings and
the topology of the moduli space of holomorphic buildings. We also
have the expected compactness theorem for the latter case. 
\begin{thm}
\label{thm:compactification} (\cite{compactness} for cylindrical
case; \cite{asympt}) For any $a\geqq0,$ the moduli space of stable
holomorphic buildings with uniformly bounded Hofer energy $E_{a}$,
whose domains have a fixed number of arithmetic genus and a fixed
number of marked points, is compact. 
\end{thm}
The following theorem shows that in the contact case Hofer energy
$E_{a}(\tilde{u})$ can be uniformly bounded by the Symplectic area
$E_{symp,a}(\tilde{u})$ and the periods of the periodic orbits of
$\mathbf{R}_{\pm\infty}$ that $\tilde{u}$ converges to at infinity
(compare to 9.2 in \cite{compactness}). 
\begin{thm}
\label{thm: hofer's energy bd by symplectic area-1} Suppose $(E,J)$
is an almost complex manifold with asymptotically cylindrical ends
of contact type. There exist positive constants $C,C',$ and $a$
such that for any finitely punctured Riemann surface $(\Sigma,j)$
and any non-constant $J$-holomorphic curve $\tilde{u}:\Sigma\to E$
which converges to periodic orbits $\gamma_{\pm}$'s of $\mathbf{R}_{\pm\infty}$
around the punctures of $\Sigma$, we have 
\[
E_{a}(\tilde{u})\leqq C\left(2\sum\int\gamma_{+}^{*}\lambda_{+\infty}-\sum\int\gamma_{-}^{*}\lambda_{-\infty}\right)+C'E_{symp,a}(\tilde{u}),
\]
 where the summations are taken over all the periodic orbits $\gamma_{\pm}$'s
of $\mathbf{R}_{\pm\infty}$ to which $\tilde{u}$ converges respectively.
\end{thm}
The proof of this theorem is given in the appendix. Roughly speaking,
it follows from Stokes' theorem.

\section{\label{sec:An-application-to}An application to closed symplectic
manifolds with a compatible $J$}

Now we would like to apply the previous results to study the moduli
space of $J$-holomorphic curves passing through a fixed point in
a closed symplectic manifold. This generalizes some results in \cite{holomorphic curves at one point}.

Let $M$ be a closed smooth symplectic manifold of dimension $2N$
with symplectic form $\omega',$ and $J$ be a compatible almost complex
structure. For a sufficiently small neighborhood $U$ of $p\in M,$
there exists a Darboux coordinate chart $\varphi:U\to B(O,\epsilon)\subseteq\mathbb{C}^{N}$
such that $\varphi(p)=O,$ $\left.\varphi_{*}J\right|_{O}=i|_{O}$
and $\varphi^{*}\omega_{st}=\omega',$ where $O$ is the origin, $B(O,\epsilon):=\left\{ \left.z\in\mathbb{C}^{N}\right||z|<\epsilon\right\} $
and $i$ is the standard complex structure on $\mathbb{C}^{N},$ and
$\omega_{st}:=\frac{i}{2}\sum_{k=1}^{n}dz_{k}\wedge d\bar{z}_{k}=\sum_{k=1}^{n}dx_{k}\wedge dy_{k}$
is the standard symplectic structure on $\mathbb{C}^{N}.$ We identify
$B(O,\epsilon)\backslash O$ with $W_{-}:=\mathbb{R}^{-}\times S^{2N-1}$
via the map $\psi(z)=(\log|z|-\log\epsilon,\frac{z}{|z|}).$ Let us
simplify the notation $\left(\psi\circ\varphi\right)_{*}J$ by $J$
when there is no confusion. This gives $(M\backslash p,J)$ the structure
of an almost complex manifold with one asymptotically cylindrical
negative end.

Indeed, we define $\xi$, $\mathbf{R},$ and $\lambda$ as before.
Then $\lambda_{-\infty}:=\underset{s\to-\infty}{\lim}f_{s}^{*}\lambda=\Pi^{*}\lambda_{st},$
where 
\[
\lambda_{st}=\left.\frac{1}{2}\sum_{k=1}^{N}\left(x_{k}dy_{k}-y_{k}dx_{k}\right)\right|_{S^{2N-1}}
\]
 is the standard contact $1$-form on the unit sphere $S^{2N-1}\subseteq\mathbb{C}^{N},$
and $\Pi:\mathbb{R}^{-}\times S^{2N-1}\to S^{2N-1}$ is the projection.
We choose $\omega_{-\infty}=d\lambda_{-\infty}.$ 

Notice that $\mathbf{R}_{-\infty}:=\underset{s\to-\infty}{\lim}f_{s}^{*}\mathbf{R}$
restricted to $S^{2N-1}$ is exactly the standard Reeb vector field
on $S^{2N-1}$, so we can see that $J$ is Morse-Bott. 

Let $(\Sigma,j)$ be a Riemann surface with finitely many punctures
and $\tilde{u}:\mbox{\ensuremath{\Sigma}}\to M\backslash p$ be a
$J$-holomorphic curve, i.e. $J(\tilde{u})\circ T\tilde{u}=T\tilde{u}\circ j$.

We say a puncture $q$ of $\Sigma$ is removable if around $q,$ $\tilde{u}$
converges to a point in $M\backslash p.$ Otherwise, we say $q$ is
non-removable. To clarify the relations between different concepts
we state the following lemma.
\begin{lem}
\label{lem:Suppose-that-all}Suppose that all the punctures of $\Sigma$
are non-removable. Then the following statements are equivalent. \end{lem}
\begin{enumerate}
\item $\tilde{u}$ converges to some Reeb orbits of $\mathbf{R}_{-\infty}$
at negative infinity around the punctures of $\Sigma$.
\item $E_{a}(\tilde{u})$ is finite for all $a\geqq0$. 
\item $E_{a}(\tilde{u})$ is finite for some $a\geqq0$.
\item $\underset{a\to+\infty}{\lim}E_{symp,a}(\tilde{u})$ is finite.
\item If we view $\tilde{u}$ as a map from $\Sigma$ to $M,$ then $\tilde{u}$
extends smoothly over $S$, where $S$ is the smooth Riemann surface
associated to $\Sigma.$\end{enumerate}
\begin{proof}
It is obvious that $(2)\Longleftrightarrow(3)$. Lemma \ref{lem: hofer energy bound for single curve}
says $(1)\Longrightarrow(3).$ From Theorem \ref{thm:converge to reeb orbit}
and Removable Singularity Theorem, we get $(3)\Longrightarrow(1)$.
$(1)\Longrightarrow(4)$ follows from direct calculation. $(4)\Longrightarrow(5)$
is true by the Removable Singularity Theorem. Finally, $(5)\Longrightarrow(1)$
is guaranteed by Theorem B%
\footnote{Theorem B is stated for the case of a $J$-holomorphic strip with
Lagrangian boundary condition, but it is easy to see that it is also
true in this closed case.%
} in \cite{Robbin Salamon}.
\end{proof}
Assuming any of the (1)-(5) is true, then by (4) and (5) we have 
\[
E_{symp}(\tilde{u})=\underset{a\to+\infty}{\lim}E_{symp,a}(\tilde{u})=\underset{a\to+\infty}{\lim}\int_{\tilde{u}^{-1}\left(E\backslash W_{-}^{a}\right)}\tilde{u}^{*}\omega'=\int_{S}\tilde{u}^{*}\omega'<+\infty.
\]
 Thus, $E(\tilde{u})=E_{symp}(\tilde{u})+E_{\omega}(\tilde{u})+E_{\lambda}(\tilde{u})$
is well defined.

The multiplicity of a Reeb orbit $\gamma$ is the degree of $\gamma$
as a cover of a simple Reeb orbit. For each non-removable punctures
$q$ of $\Sigma,$ we can associate a positive integer which is the
multiplicity of the corresponding Reeb orbit that $\tilde{u}$ converges
to around $q$.

Let $\tilde{u}$ be a non-constant $J$-holomorphic curve from a smooth
Riemann surface $(S,j)$ to $M.$ By the Carleman Similarity principle,
we know $\tilde{u}^{-1}(p)$ is discrete, and hence finite. Let $(\Sigma,j)$
be the punctured Riemann surface $(S\backslash\tilde{u}^{-1}(p),j).$
Now $\tilde{u}$ can be viewed as a $J$-holomorphic curve from $\Sigma$
to $M\backslash p.$ This means that the condition (4) in Lemma \ref{lem:Suppose-that-all}
is satisfied, so we have (1)-(5). An easy modification of the proof
of Theorem \ref{thm: hofer's energy bd by symplectic area-1} leads
to the next theorem. 
\begin{thm}
(Gromov's Monotonicity Theorem with multiplicity\label{thm:(Gromov's-Monotonicity-Theorem})
For a closed symplectic manifold $(M,\omega')$ with a compatible
almost complex structure $J,$ there exists a constant $r_{0}>0$
and a function $\hbar(r)>0$ such that for any point $p\in M,$ and
any $J$-holomorphic curve $\tilde{u}$ from a Riemann surface (with
boundary) $\mathcal{S}$ mapped to $M$ that passes through the point
$p$ for $k$ times (counted with multiplicity), and satisfies $\tilde{u}(\partial\mathcal{S})\cap B_{r}(p)=\emptyset,$
for $0<r<r_{0},$ the following is true. 
\[
\int_{\tilde{u}^{-1}(B_{r}(p))}\tilde{u}^{*}\omega'>k\hbar(r),
\]
 where $B_{r}(p)$ is a ball of radius $r$ centered at $p$ inside
$M.$
\end{thm}
The proof of Theorem \ref{thm:(Gromov's-Monotonicity-Theorem} is
given in the appendix. Now it follows immediately that 
\begin{cor}
\label{cro: closed case hofer's energy bd by symplectic area} There
exists a constant $C>0$ depending only on $(M,\omega,J)$ such that
for any Riemann surface $(S,j)$ and any non-constant $J$-holomorphic
curve $\tilde{u}:S\to M$ passing through a point $p$ for $k$ times,
we have $k\leqq CE_{symp}(\tilde{u}).$
\end{cor}

\begin{rem}
\label{rmk: Fish}After the submission of the arXiv version 1 of this
paper, we were informed that Corollary \ref{cro: closed case hofer's energy bd by symplectic area}
could also be derived from Corollary 3.6 and the remarks below Corollary
3.6 in \cite{fish}. It is very interesting to see that the methods
used in \cite{fish} and this paper are quite different. In \cite{fish}
the technics from minimal surfaces is used, and a stronger result
than Theorem \ref{thm:(Gromov's-Monotonicity-Theorem} is achieved.
In particular, \cite{fish} implies that $\hbar(r)$ is proportional
to $r^{2}.$ While, in this paper, we view $M\backslash p$ as a manifold
with asymptotically end, and Theorem \ref{thm:(Gromov's-Monotonicity-Theorem}
follows roughly from Stokes' Theorem immediately. (Also see \cite{holomorphic curves at one point}
for a slightly different proof.) However, using this method it is
not clear why $\hbar(r)$ is proportional to $r^{2}.$
\end{rem}

Let $\mathcal{M}_{g}(M,J,Q)$ be the moduli space of stable $J$-holomorphic
curves $\tilde{u}$ in $M$ with genus $g$ and $E_{symp}(\tilde{u})\leqq Q.$
From Corollary \ref{cro: closed case hofer's energy bd by symplectic area}
and Theorem \ref{thm:compactification}, we can compactify $\mathcal{M}_{g}(M,J,Q)$
by including holomorphic buildings (See \cite{holomorphic curves at one point}
for more discussions). 

It will be very interesting and useful to generalize the results in
this paper by replacing the fixed point $p$ with an almost complex
submanifold.

\section{Appendix: Proof of Theorem \ref{thm: hofer's energy bd by symplectic area-1}
and Theorem \ref{thm:(Gromov's-Monotonicity-Theorem}}

For convenience let us introduce the following terminology.
\begin{defn}
We say that a $2-$form $\Delta$ defined on $(-\infty,-R]\times V_{-}$
is $J-$positive (or non-negative), if for a sufficiently large $R,$
$\Delta$ is positive (or non-negative) on any $J-$complex planes
of $TW_{-R}:=T\left((-\infty,-R]\times V_{-}\right).$ In other words,
$\Delta(x,Jx)>0$ (or $\geqq0$), for all $x\in TW_{-R}.$
\end{defn}

\begin{defn}
We say that a $2-$form $\Delta$ defined on $(-\infty,-R]\times V_{-}$
is $J-$positively bounded away from $0,$ if $\inf\Delta(x,Jx)>0$,
where the infimum is taken over all the $x\in TW_{-R}.$ with norm
$\left\Vert x\right\Vert _{g_{W_{-}}}=1$ (Recall that $g_{W_{-}}$
is a translational invariant metric). \end{defn}
\begin{proof}
(Theorem \ref{thm: hofer's energy bd by symplectic area-1}) Let us
deal with the negative end $W_{-}$ first.

For any $R>0,$ we pick $-\mathfrak{r}\in\left[-2R,-R\right]$ such
that $-\mathfrak{r}$ is a regular value of $r\circ\tilde{u},$ where
$r:W_{-}\to(-\infty,0)$ is the projection map. Denote $A:=\tilde{u}^{-1}((-\infty,-\mathfrak{r}]\times V_{-})\subseteq\Sigma$
and $B_{1}:=\tilde{u}^{-1}(\{-\mathfrak{r}\}\times V_{-})$. Let $\hat{A}$
be the oriented blow up of $A$ around all the punctures of $A$,
i.e. $\hat{A}=A\sqcup B_{2}$ with $B_{2}:=\sqcup S^{1}$ being the
disjoint union of circles introduced by the oriented blow up. Hence
we have $\partial\hat{A}=B_{1}\sqcup B_{2}.$ We choose the orientation
of $B_{1}$ to be the boundary orientation from $\hat{A},$ while
we choose the orientation of $B_{2}$ to be the reverse orientation
of the boundary orientation from $\hat{A}.$ 

For $x\in TW_{-}=\mathbb{R}(\frac{\partial}{\partial r})\oplus\mathbb{R}(\mathbf{R}_{-\infty})\oplus\xi_{-\infty},$
we can write $x$ as $x=dr(x)\frac{\partial}{\partial r}+\lambda(x)\mathbf{R}_{-\infty}+\pi_{\xi_{-\infty}}x.$
Then for any constants $P,Q>0,$ we have 
\begin{eqnarray*}
 &  & \left[Pd\lambda_{-\infty}+Qdr\wedge\lambda_{-\infty}\right](x,J_{-\infty}x)\\
 & = & Pd\lambda_{-\infty}(\pi_{\xi_{-\infty}}x,J_{-\infty}\pi_{\xi_{-\infty}}x)+Q\left[dr(x)\right]^{2}+Q\left[\lambda(x)\right]^{2}.
\end{eqnarray*}
 Because $d\lambda_{-\infty}(\cdot,J_{-\infty}\cdot)$ defines a metric
on $\xi_{-\infty}$, we get $Pd\lambda_{-\infty}+Qdr\wedge\lambda_{-\infty}$
is $J_{-\infty}-$positively bounded away from $0$. Denote 
\[
\mathcal{S}:=\left\{ \left.(x,y)\in TW_{-}\times TW_{-}\right|\left\Vert x\right\Vert _{g_{W_{-}}}=1,y=J_{-\infty}x\right\} 
\]
 and 
\[
\mathcal{T}_{-R}:=\left\{ \left.(x,y)\in TW_{-R}\times TW_{-R}\right|\left\Vert x\right\Vert _{g_{W_{-}}}=1,y=Jx\right\} .
\]
 Let $\Delta$ be the smooth map $TW_{-}\times TW_{-}\to\mathbb{R}$
defined by applying $Pd\lambda_{-\infty}+Qdr\wedge\lambda_{-\infty}.$
The fact that $Pd\lambda_{-\infty}+Qdr\wedge\lambda_{-\infty}$ is
$J_{-\infty}-$positively bounded away from $0$ means that $\Delta|_{\mathcal{S}}>\varpi>0$
for some enough small $\varpi.$ By (ACC1) there exists $R$ large
enough, such that $\Delta|_{\mathcal{T}_{-R}}>\frac{1}{2}\varpi>0.$
Therefore, we get that $Pd\lambda_{-\infty}+Qdr\wedge\lambda_{-\infty}$
is $J-$positively bounded away from $0$. 

Since $J-J_{-\infty}$ is exponentially small by (ACC1), there exist
constants $C_{1},\kappa_{1}>0$ such that 
\begin{equation}
\left|d\lambda_{-\infty}(x,(J-J_{-\infty})x)\right|\leqq\frac{1}{2}C_{1}e^{\kappa_{1}r},\label{eq: d lambda J-J infty}
\end{equation}
 for all $x\in TW_{-R}$ with $||x||_{g_{W_{-}}}=1.$ 

From now on let us pick $g_{W_{-}}$ to be $\left\langle x,y\right\rangle _{g_{W_{-}}}=\left(dr\wedge\lambda_{-\infty}+d\lambda_{-\infty}\right)(x,J_{-\infty}y)$
for convenience. Notice that by (ACC1) again, for all $x\in TW_{-R}$
with $||x||_{g_{W_{-}}}=1,$ we get 
\begin{eqnarray}
 &  & \left(dr\wedge\lambda_{-\infty}+d\lambda_{-\infty}\right)(x,Jx)\nonumber \\
 & = & \left(dr\wedge\lambda_{-\infty}+d\lambda_{-\infty}\right)(x,J_{-\infty}x)+\left(dr\wedge\lambda_{-\infty}+d\lambda_{-\infty}\right)(x,(J-J_{-\infty})x)\nonumber \\
 & \geqq & \left\Vert x\right\Vert _{g_{W_{-}}}-K_{0}e^{\delta_{0}r}\nonumber \\
 & > & \frac{1}{2}.\label{eq:>0.5}
\end{eqnarray}
 Hence (\ref{eq: d lambda J-J infty}) and (\ref{eq:>0.5}) imply

\begin{eqnarray*}
 &  & \left[d\lambda_{-\infty}+C_{1}e^{\kappa_{1}r}\left(dr\wedge\lambda_{-\infty}+d\lambda_{-\infty}\right)\right](x,Jx)\\
 & > & d\lambda_{-\infty}(x,J_{-\infty}x)+d\lambda_{-\infty}(x,(J-J_{-\infty})x)+\frac{1}{2}C_{1}e^{\kappa_{1}r}\\
 & \geqq & d\lambda_{-\infty}(x,J_{-\infty}x)\\
 & \geqq & 0,
\end{eqnarray*}
 where the last inequality comes from the fact that $d\lambda_{-\infty}$
is $J_{-\infty}-$non-negative. Hence

\[
d\lambda_{-\infty}+C_{1}e^{\kappa_{1}r}\left(dr\wedge\lambda_{-\infty}+d\lambda_{-\infty}\right)
\]
 is $J-$positive, so is 
\[
d\lambda_{-\infty}+\frac{C_{1}e^{\kappa_{1}r}}{1+C_{1}e^{\kappa_{1}r}}dr\wedge\lambda_{-\infty}.
\]

Similarly, by varying $C_{1}$ and $\kappa_{1}$ if necessary, we
can get that 
\[
\left|dr\wedge\lambda_{-\infty}(x,(J-J_{-\infty})x)\right|\leqq\frac{1}{2}C_{1}e^{\kappa_{1}r},
\]
 for all $x\in TW_{-R}$ with $||x||_{g_{W_{-}}}=1.$ As before, we
have 
\begin{eqnarray*}
 &  & \left[dr\wedge\lambda_{-\infty}+C_{1}e^{\kappa_{1}r}\left(dr\wedge\lambda_{-\infty}+d\lambda_{-\infty}\right)\right](x,Jx)\\
 & > & dr\wedge\lambda_{-\infty}(x,J_{-\infty}x)+dr\wedge\lambda_{-\infty}(x,(J-J_{-\infty})x)+\frac{1}{2}C_{1}e^{\kappa_{1}r}\\
 & \geqq & dr\wedge\lambda_{-\infty}(x,J_{-\infty}x)\\
 & \geqq & 0.
\end{eqnarray*}
 This implies 
\[
dr\wedge\lambda_{-\infty}+C_{1}e^{\kappa_{1}r}\left(dr\wedge\lambda_{-\infty}+d\lambda_{-\infty}\right)
\]
 is $J$-positive, so is

\[
dr\wedge\lambda_{-\infty}+\frac{C_{1}e^{\kappa_{1}r}}{1+C_{1}e^{\kappa_{1}r}}d\lambda_{-\infty}.
\]

From equation (\ref{eq:omega-omega -infty is small}) and the fact
$\omega_{-\infty}=d\lambda_{-\infty},$ we get $\omega-d\lambda_{-\infty}$
is exponentially small. Because $J-J_{-\infty}$ is also exponentially
small, by varying $C_{1}$ and $\kappa_{1}$ if necessary, we can
get that 
\begin{equation}
\left|\left(\omega-d\lambda_{-\infty}\right)(x,Jx)\right|\leqq\frac{1}{2}C_{1}e^{\kappa_{1}r},\label{eq:omega-d lambda infty}
\end{equation}
for all $x\in TW_{-R}$ with $||x||_{g_{W_{-}}}=1.$ 

Therefore, by (\ref{eq:omega-omega -infty is small}) and (\ref{eq:>0.5})
we have 

\[
\omega(x,Jx)\leqq d\lambda_{-\infty}(x,Jx)+C_{1}e^{\kappa_{1}r}\left(dr\wedge\lambda_{-\infty}+d\lambda_{-\infty}\right)(x,Jx).
\]
 When restricted to $J$-complex planes in $TW_{-R}$ for large $R,$
we get 

\begin{eqnarray}
\omega & \leqq & d\lambda_{-\infty}+C_{1}e^{\kappa_{1}r}\left(dr\wedge\lambda_{-\infty}+d\lambda_{-\infty}\right)\nonumber \\
 & \leqq & (1+C_{1}e^{\kappa_{1}r})\left(d\lambda_{-\infty}+\frac{C_{1}e^{\kappa_{1}r}}{1+C_{1}e^{\kappa_{1}r}}dr\wedge\lambda_{-\infty}\right)\nonumber \\
 & \leqq & C_{2}\left(d\lambda_{-\infty}+\frac{C_{1}e^{\kappa_{1}r}}{1+C_{1}e^{\kappa_{1}r}}dr\wedge\lambda_{-\infty}\right),\label{eq:omega<}
\end{eqnarray}
where $C_{2}=1+C_{1}e^{-\kappa_{1}R}<2.$

Similarly, for all $x\in TW_{-R}$ with $||x||_{g_{W_{-}}}=1,$ we
have

\begin{equation}
\left|\left(\sigma\wedge\lambda-dr\wedge\lambda_{-\infty}\right)(x,Jx)\right|\leqq\frac{1}{2}C_{1}e^{\kappa_{1}r}.\label{eq:sigma wedge lambda - dr wedge lambda infty}
\end{equation}
Hence when restricted to $J$-complex planes in $TW_{-R}$ for large
$R,$ by (\ref{eq:>0.5}) and (\ref{eq:sigma wedge lambda - dr wedge lambda infty})
we have 
\begin{eqnarray}
\sigma\wedge\lambda & \leqq & dr\wedge\lambda_{-\infty}+C_{1}e^{\kappa_{1}r}\left(dr\wedge\lambda_{-\infty}+d\lambda_{-\infty}\right)\nonumber \\
 & \leqq & C_{2}\left(\frac{C_{1}e^{\kappa_{1}r}}{1+C_{1}e^{\kappa_{1}r}}d\lambda_{-\infty}+dr\wedge\lambda_{-\infty}\right).\label{eq:sigma lambda}
\end{eqnarray}

On the other hand, since $\omega+\sigma\wedge\lambda$ is $J$ - positively
bounded away from $0,$ when restricted on $J$-complex planes in
$TW_{-R}$ for large $R,$ we get 
\begin{eqnarray}
\left|dr\wedge\lambda_{-\infty}\right| & \leqq & \left|dr\wedge\lambda_{-\infty}-\sigma\wedge\lambda\right|+\sigma\wedge\lambda\nonumber \\
 & \leqq & C_{1}e^{\kappa_{1}r}\left(\omega+\sigma\wedge\lambda\right)+\sigma\wedge\lambda\nonumber \\
 & \leqq & C_{1}e^{\kappa_{1}r}\omega+C_{2}\sigma\wedge\lambda\label{eq:dr lambda infty}
\end{eqnarray}
 and 
\begin{eqnarray}
\left|d\lambda_{-\infty}\right| & \leqq & \left|d\lambda_{-\infty}-\omega\right|+\omega\nonumber \\
 & \leqq & C_{1}e^{\kappa_{1}r}\left(\omega+\sigma\wedge\lambda\right)+\omega\nonumber \\
 & \leqq & C_{2}\omega+C_{1}e^{\kappa_{1}r}\sigma\wedge\lambda,\label{eq:d lambda infty}
\end{eqnarray}
 by modifying $C_{1}$ and $\kappa_{1}.$

Therefore, we have 

\begin{eqnarray}
 &  & \int_{\tilde{u}^{-1}(W_{-})}\tilde{u}^{*}\omega\nonumber \\
 & \leqq & \int_{\hat{A}}\tilde{u}^{*}\omega+\int_{\{\Sigma\backslash A\}\cap\tilde{u}^{-1}(W_{-})}\tilde{u}^{*}\omega\nonumber \\
 & \leqq & C_{2}\int_{\hat{A}}\tilde{u}^{*}\left(d\lambda_{-\infty}+\frac{C_{1}e^{\kappa_{1}r}}{1+C_{1}e^{\kappa_{1}r}}dr\wedge\lambda_{-\infty}\right)+\int_{\{\Sigma\backslash A\}\cap\tilde{u}^{-1}(W_{-})}\tilde{u}^{*}\omega\nonumber \\
 & = & C_{2}\int_{B_{1}}\tilde{u}^{*}\lambda_{-\infty}-C_{2}\int_{B_{2}}\tilde{u}^{*}\lambda_{-\infty}\nonumber \\
 &  & +C_{2}\int_{\hat{A}}\tilde{u}^{*}\left(\frac{C_{1}e^{\kappa_{1}r}}{1+C_{1}e^{\kappa_{1}r}}dr\wedge\lambda_{-\infty}\right)+\int_{\{\Sigma\backslash A\}\cap\tilde{u}^{-1}(W_{-})}\tilde{u}^{*}\omega.\label{eq:omega-1}
\end{eqnarray}

While,

\begin{eqnarray}
 &  & \left|\int_{\hat{A}}\tilde{u}^{*}\left(\frac{C_{1}e^{\kappa_{1}r}}{1+C_{1}e^{\kappa_{1}r}}dr\wedge\lambda_{-\infty}\right)\right|\nonumber \\
 & \leqq & \int_{\hat{A}}\left|\tilde{u}^{*}\left(\frac{C_{1}e^{\kappa_{1}r}}{1+C_{1}e^{\kappa_{1}r}}dr\wedge\lambda_{-\infty}\right)\right|\nonumber \\
 & \leqq & C_{1}\int_{\hat{A}}\left|\tilde{u}^{*}e^{\kappa_{1}r}\left(C_{1}e^{\kappa_{1}r}\omega+C_{2}\sigma\wedge\lambda\right)\right|\nonumber \\
 & \leqq & \frac{1}{4}E_{\omega}(\tilde{u}|_{W_{-}})+C_{1}C_{2}\kappa_{1}^{-1}e^{-\kappa_{1}\mathfrak{r}}\int_{\hat{A}}\tilde{u}^{*}\left(\kappa_{1}e^{\kappa_{1}(\mathfrak{r}+r)}\sigma\wedge\lambda\right).\label{eq:error error}
\end{eqnarray}
 Since $\int_{-\infty}^{-\mathfrak{r}}\kappa_{1}e^{\kappa_{1}(\mathfrak{r}+r)}dr=1,$
we have 
\[
\int_{\hat{A}}\tilde{u}^{*}\left(\kappa_{1}e^{\kappa_{1}(\mathfrak{r}+r)}\sigma\wedge\lambda\right)\leqq E_{\lambda}(\tilde{u}).
\]
 Therefore, by picking $R$ sufficiently large, we can make $\mathfrak{r}$
sufficiently large, and then (\ref{eq:error error}) implies 

\begin{equation}
\left|\int_{\hat{A}}\tilde{u}^{*}\left(\frac{C_{1}e^{\delta_{1}r}}{1+C_{1}e^{\delta_{1}r}}dr\wedge\lambda_{-\infty}\right)\right|\leqq\frac{1}{4}E_{\omega}(\tilde{u}|_{W_{-}})+\frac{1}{4}E_{\lambda}(\tilde{u}|_{W_{-}}).\label{eq:omega error}
\end{equation}

Let $\Phi(r)=\intop_{-\infty}^{r}\phi(t)dt,$ for $\phi\in\mathcal{C},$
and then we get

\begin{eqnarray*}
 &  & \int_{\tilde{u}^{-1}(W_{-})}\tilde{u}^{*}\left(\phi(r)\sigma\wedge\lambda\right)\\
 & \leqq & \int_{\hat{A}}\tilde{u}^{*}\left(\phi(r)\sigma\wedge\lambda\right)+\int_{\{\Sigma\backslash A\}\cap\tilde{u}^{-1}(W_{-})}\tilde{u}^{*}\left(\phi(r)\sigma\wedge\lambda\right)\\
 & \leqq & C_{2}\int_{\hat{A}}\tilde{u}^{*}\left(\phi(r)dr\wedge\lambda_{-\infty}+\frac{C_{1}e^{\kappa_{1}r}}{1+C_{1}e^{\kappa_{1}r}}\phi(r)d\lambda_{-\infty}\right)\\
 &  & +\int_{\{\Sigma\backslash A\}\cap\tilde{u}^{-1}(W_{-})}\tilde{u}^{*}\left(\phi(r)\sigma\wedge\lambda\right)\\
 & \leqq & C_{2}\int_{\hat{A}}\tilde{u}^{*}d(\Phi(r)\lambda_{-\infty})-C_{2}\int_{\hat{A}}\tilde{u}^{*}(\Phi(r)d\lambda_{-\infty})\\
 &  & +C_{2}C_{1}\int_{\hat{A}}\tilde{u}^{*}\left(\frac{e^{\kappa_{1}r}}{1+C_{1}e^{\kappa_{1}r}}\phi(r)d\lambda_{-\infty}\right)+\int_{\{\Sigma\backslash A\}\cap\tilde{u}^{-1}(W_{-})}\tilde{u}^{*}\left(\phi(r)\sigma\wedge\lambda\right)
\end{eqnarray*}

\begin{eqnarray}
 & \leqq & C_{2}\int_{\hat{A}}\tilde{u}^{*}d(\Phi(r)\lambda_{-\infty})-C_{2}\left\{ \int_{\hat{A}}\tilde{u}^{*}\left[\Phi(r)d\lambda_{-\infty}+\Phi(r)\frac{C_{1}e^{\kappa_{1}r}}{1+C_{1}e^{\kappa_{1}r}}dr\wedge\lambda_{-\infty}\right]\right\} \nonumber \\
 &  & +C_{2}C_{1}\int_{\hat{A}}\tilde{u}^{*}\left(\Phi(r)\frac{e^{\kappa_{1}r}}{1+C_{1}e^{\kappa_{1}r}}dr\wedge\lambda_{-\infty}\right)\nonumber \\
 &  & +C_{2}C_{1}\int_{\hat{A}}\tilde{u}^{*}\left(\frac{e^{\kappa_{1}r}}{1+C_{1}e^{\kappa_{1}r}}\phi(r)d\lambda_{-\infty}\right)\nonumber \\
 &  & +\int_{\{\Sigma\backslash A\}\cap\tilde{u}^{-1}(W_{-})}\tilde{u}^{*}\left(\phi(r)\sigma\wedge\lambda\right)\nonumber \\
 & \leqq & C_{2}\int_{\hat{A}}\tilde{u}^{*}d(\Phi(r)\lambda_{-\infty})+C_{2}C_{1}\int_{\hat{A}}\tilde{u}^{*}\left(\Phi(r)\frac{e^{\kappa_{1}r}}{1+C_{1}e^{\kappa_{1}r}}dr\wedge\lambda_{-\infty}\right)\nonumber \\
 &  & +C_{2}C_{1}\int_{\hat{A}}\tilde{u}^{*}\left(\frac{e^{\kappa_{1}r}}{1+C_{1}e^{\kappa_{1}r}}\phi(r)d\lambda_{-\infty}\right)+\int_{\{\Sigma\backslash A\}\cap\tilde{u}^{-1}(W)}\tilde{u}^{*}\left(\phi(r)\sigma\wedge\lambda\right)\nonumber \\
 & = & C_{2}\int_{B_{1}}\tilde{u}^{*}(\Phi(-\mathfrak{r})\lambda_{-\infty})+C_{2}C_{1}\int_{\hat{A}}\tilde{u}^{*}\left(\Phi(r)\frac{e^{\kappa_{1}r}}{1+C_{1}e^{\kappa_{1}r}}dr\wedge\lambda_{-\infty}\right)\nonumber \\
 &  & +C_{2}C_{1}\int_{\hat{A}}\tilde{u}^{*}\left(\frac{e^{\kappa_{1}r}}{1+C_{1}e^{\kappa_{1}r}}\phi(r)d\lambda_{-\infty}\right)\nonumber \\
 &  & +\int_{\{\Sigma\backslash A\}\cap\tilde{u}^{-1}(W_{-})}\tilde{u}^{*}\left(\phi(r)\sigma\wedge\lambda\right),\label{eq:sigma wedge lambda}
\end{eqnarray}
where the last inequality follows from the fact that $\Phi(r)d\lambda_{-\infty}+\Phi(r)\frac{C_{1}e^{\kappa_{1}r}}{1+C_{1}e^{\kappa_{1}r}}dr\wedge\lambda_{-\infty}$
is $J$- positive. 

While we have 

\begin{eqnarray}
 &  & \left|C_{2}C_{1}\int_{\hat{A}}\tilde{u}^{*}\left(\Phi(r)\frac{e^{\kappa_{1}r}}{1+C_{1}e^{\kappa_{1}r}}dr\wedge\lambda_{-\infty}\right)\right|\nonumber \\
 & \leqq & C_{2}C_{1}\int_{\hat{A}}\tilde{u}^{*}\left|e^{\kappa_{1}r}dr\wedge\lambda_{-\infty}\right|\label{eq:sigma wedge lambda error1}\\
 & \leqq & \frac{1}{4}E_{\omega}(\tilde{u}|_{W_{-}})+\frac{1}{4}E_{\lambda}(\tilde{u}|_{W_{-}}),\nonumber 
\end{eqnarray}
 and 

\begin{eqnarray}
 &  & \left|C_{2}C_{1}\int_{\hat{A}}\tilde{u}^{*}\left(\frac{e^{\kappa_{1}r}}{1+C_{1}e^{\kappa_{1}r}}\phi(r)d\lambda_{-\infty}\right)\right|\nonumber \\
 & \leqq & C_{2}C_{1}\int_{\hat{A}}\tilde{u}^{*}e^{\kappa_{1}r}\left(C_{2}\omega+C_{1}e^{\kappa_{1}r}\sigma\wedge\lambda\right)\label{eq:sigma wedge lambda error 2}\\
 & \leqq & \frac{1}{4}E_{\omega}(\tilde{u}|_{W_{-}})+\frac{1}{4}E_{\lambda}(\tilde{u}|_{W_{-}}).\nonumber 
\end{eqnarray}

Therefore, from (\ref{eq:omega-1}), (\ref{eq:omega error}), (\ref{eq:sigma wedge lambda}),
(\ref{eq:sigma wedge lambda error1}), and (\ref{eq:sigma wedge lambda error 2}),
we get 
\begin{eqnarray*}
E(\tilde{u}|_{W_{-}}) & := & E_{\omega}(\tilde{u}|_{W_{-}})+E_{\lambda}(\tilde{u}|_{W_{-}})\\
 & \leqq & 2C_{2}\int_{B_{1}}\tilde{u}^{*}\lambda_{-\infty}-C_{2}\int_{B_{2}}\tilde{u}^{*}\lambda_{-\infty}+\frac{3}{4}E_{\omega}(\tilde{u}|_{W_{-}})+\frac{3}{4}E_{\lambda}(\tilde{u}|_{W_{-}})\\
 &  & +\int_{\{\Sigma\backslash A\}\cap\tilde{u}^{-1}(W_{-})}\tilde{u}^{*}\omega+\int_{\{\Sigma\backslash A\}\cap\tilde{u}^{-1}(W_{-})}\tilde{u}^{*}\left(\phi(r)\sigma\wedge\lambda\right).
\end{eqnarray*}
Thus, 
\begin{eqnarray}
E(\tilde{u}|_{W_{-}}) & \leqq & 4C_{2}\left(2\int_{B_{1}}\tilde{u}^{*}\lambda_{-\infty}-\int_{B_{2}}\tilde{u}^{*}\lambda_{-\infty}\right)+4\int_{\{\Sigma\backslash A\}\cap\tilde{u}^{-1}(W_{-})}\tilde{u}^{*}\omega\nonumber \\
 &  & +4\int_{\{\Sigma\backslash A\}\cap\tilde{u}^{-1}(W_{-})}\tilde{u}^{*}\left(\phi(r)\sigma\wedge\lambda\right).\label{eq:E<}
\end{eqnarray}
Now we define a function $\tau$ by $\tau(r)=\frac{R+r}{R-\mathfrak{r}}$
for $-\mathfrak{r}\leqq r\leqq-R.$ Since $\tau(-\mathfrak{r})=1$
and $\tau(-R)=0,$ by Stokes' Theorem we get 
\begin{eqnarray}
\left|\intop_{B_{1}}\tilde{u}^{*}\lambda_{-\infty}\right| & = & \left|\int_{\{\Sigma\backslash A\}\cap\tilde{u}^{-1}([-\mathfrak{r},-R]\times V_{-})}\tilde{u}^{*}d\left(\tau(r)\lambda_{-\infty}\right)\right|\nonumber \\
 & \leqq & \int_{\{\Sigma\backslash A\}\cap\tilde{u}^{-1}([-\mathfrak{r},-R]\times V_{-})}\left|\tilde{u}^{*}d\left(\tau(r)\lambda_{-\infty}\right)\right|\nonumber \\
 & \leqq & C_{3}\int_{\{\Sigma\backslash A\}\cap\tilde{u}^{-1}([-\mathfrak{r},-R]\times V_{-})}\tilde{u}^{*}\omega'\label{eq:lambda infinity interior}\\
 & \leqq & C_{3}\int_{\{\Sigma\backslash A\}\cap\tilde{u}^{-1}(W_{-})}\tilde{u}^{*}\omega'\label{eq:lambda-3}
\end{eqnarray}
where $C_{3}$ is a constant depending on $R$, and the second inequality
follows from the fact that on any $J$-complex planes the symplectic
form $\omega'$ is positive. For the same reason, by modifying $C_{3}$
if necessary, we also have 
\begin{equation}
\int_{\{\Sigma\backslash A\}\cap\tilde{u}^{-1}(W_{-})}\tilde{u}^{*}\omega\leqq C_{3}\int_{\{\Sigma\backslash A\}\cap\tilde{u}^{-1}(W_{-})}\tilde{u}^{*}\omega'\label{eq:omega-2}
\end{equation}
 and 
\begin{equation}
\int_{\{\Sigma\backslash A\}\cap\tilde{u}^{-1}(W_{-})}\tilde{u}^{*}\left(\sigma\wedge\lambda\right)\leqq C_{3}\int_{\{\Sigma\backslash A\}\cap\tilde{u}^{-1}(W_{-})}\tilde{u}^{*}\omega'.\label{eq:sigma lambda -1}
\end{equation}
Then (\ref{eq:E<}), (\ref{eq:lambda-3}), (\ref{eq:omega-2}), (\ref{eq:sigma lambda -1}),
and $\int_{\{\Sigma\backslash A\}\cap\tilde{u}^{-1}(W_{-})}\tilde{u}^{*}\omega'\leqq E_{symp,2R}(\tilde{u})$
together imply 

\begin{equation}
E(\tilde{u}|_{W_{-}})\leqq C_{4}E_{symp,2R}(\tilde{u})-4C_{2}\sum\int\gamma_{-}^{*}\lambda_{-\infty},\label{eq:positive hofer energy}
\end{equation}
 where $C_{4}=8(C_{2}+1)C_{3}$ is a constant independent of $\tilde{u},$
and the summation is taken over all the periodic orbits $\gamma_{-}$'s
of $\mathbf{R}_{-\infty}$ to which $\tilde{u}$ converges at negative
infinity.

For positive end $W_{+},$ the estimates are very similar. The only
difference comes from the fact that the orientation of $V_{+}$ agrees
with the boundary orientation of $\{+\infty\}\times V_{+},$ and the
orientation of $V_{-}$ disagrees with the boundary orientation of
$\{-\infty\}\times V_{-}.$ One can easily adjust the above estimates
to $W_{+}$ case. For example, in (\ref{eq:omega-1}) the main part
is $\int_{\hat{A}}\tilde{u}^{*}d\lambda_{-\infty}=\int_{B_{1}}\tilde{u}^{*}\lambda_{-\infty}-\int_{B_{2}}\tilde{u}^{*}\lambda_{-\infty},$
and in $W_{+}$-version we replace it by 
\begin{eqnarray*}
\int_{\tilde{u}^{-1}(\hat{A}_{+})}\tilde{u}^{*}d\lambda_{+\infty} & = & \int_{B_{2+}}\tilde{u}^{*}\lambda_{+\infty}-\int_{B_{1+}}\tilde{u}^{*}\lambda_{+\infty},
\end{eqnarray*}
 where $B_{1+}:=\tilde{u}^{-1}(\{\mathfrak{r}_{+}\}\times V_{+})$
and $B_{2+}:=\tilde{u}^{-1}(\{+\infty\}\times V_{+});$ in (\ref{eq:sigma wedge lambda})
the main part is $\int_{\hat{A}}\tilde{u}^{*}d(\Phi(r)\lambda_{-\infty})=\int_{B_{1}}\tilde{u}^{*}(\Phi(-\mathfrak{r})\lambda_{-\infty}),$
and in $W_{+}$-version we replace it by 
\[
\int_{\tilde{u}^{-1}(\hat{A}_{+})}\tilde{u}^{*}d(\Phi_{+}(r)\lambda_{+\infty})\leqq\int_{B_{2+}}\tilde{u}^{*}\lambda_{+\infty}-\int_{B_{1+}}\tilde{u}^{*}(\Phi_{+}(\mathfrak{r}_{+})\lambda_{+\infty}).
\]
 Then a similar estimate as in (\ref{eq:lambda-3}) shows that all
the error terms including $-\int_{B_{1+}}\tilde{u}^{*}\lambda_{+\infty}$
and $-\int_{B_{1+}}\tilde{u}^{*}(\Phi_{+}(\mathfrak{r}_{+})\lambda_{+\infty})$
can be bounded by a multiple of $E_{symp,2R}(\tilde{u}).$ Indeed,
one can show that

\begin{equation}
E(\tilde{u}|_{W_{+}})=E_{\omega}(\tilde{u}|_{W_{+}})+E_{\lambda}(\tilde{u}|_{W_{+}})\leqq8C_{2}\sum\intop\gamma_{+}^{*}\lambda_{+\infty}+C_{4}E_{symp,2R}(\tilde{u}),\label{eq:negative hofer energy}
\end{equation}
where the summations are taken over all the periodic orbits $\gamma_{+}$'s
of $\mathbf{R}_{+\infty}$ to which $\tilde{u}$ converges at positive
infinity. 

By (\ref{eq:negative hofer energy}) and (\ref{eq:positive hofer energy}),
we have $E_{a}(\tilde{u})\leqq C\left(2\sum\intop\gamma_{+}^{*}\lambda_{+\infty}-\sum\int\gamma_{-}^{*}\lambda_{-\infty}\right)+C'E_{symp,a}(\tilde{u}),$
where $C=4C_{2}$ and $C'=2C_{4}.$ 
\end{proof}

\begin{proof}
(Theorem \ref{thm:(Gromov's-Monotonicity-Theorem}) We view $(M\backslash p,J)$
as an almost complex manifold with asymptotically cylindrical negative
end $W_{-}$ as described in the beginning of Section \ref{sec:Almost-complex-manifolds with ends},
with $W_{-}$ biholomorphic to $B_{r}(p)\backslash\{p\}.$ Notice
that all the estimates before formula (\ref{eq:lambda infinity interior})
in the proof of Theorem \ref{thm: hofer's energy bd by symplectic area-1}
are local, i.e. inside $W_{-}.$ Thus, we get 
\[
E(\tilde{u}|_{W_{-}})\leqq C_{4}\int_{\tilde{u}^{-1}(W_{-})}\tilde{u}^{*}\omega'-4C_{2}\sum\int\gamma_{-}^{*}\lambda_{-\infty}=C_{4}\int_{\tilde{u}^{-1}(W_{-})}\tilde{u}^{*}\omega'-4C_{2}(2k\pi).
\]
 From the fact that $E(\tilde{u}|_{W_{-}})>0,$ we get 
\[
\int_{\tilde{u}^{-1}(B_{r}(p))}\tilde{u}^{*}\omega'=\int_{\tilde{u}^{-1}(W_{-})}\tilde{u}^{*}\omega'>\frac{4C_{2}(2k\pi)}{C_{4}}.
\]
 Now we show that the constant $\frac{4C_{2}(2\pi)}{C_{4}}$ can be
chosen to be independent of $p.$ For each point $p\in M,$ we can
choose a Darboux chart whose size is uniformly bounded away from $0$
and the almost complex structure $J$ at $p$ coincides with the standard
one defined by $\frac{\partial}{\partial x}\mapsto\frac{\partial}{\partial y}$
and $\frac{\partial}{\partial y}\mapsto-\frac{\partial}{\partial x}.$
Identifying this neighborhood minus $p$ with the half infinite cylinder
as described in the beginning of Section \ref{sec:An-application-to},
we get (ACC1)-(ACC5) are satisfied with constants $K_{l}$ bounded
by the $C^{l}$-norm of $J$ and the norm of the Nijenhuis tensor
of $J$ (We only need $K_{0}$ in this theorem). Since we assume that
$M$ is compact, and $\omega'$ and $J$ are smooth, we can make $K_{l}$
independent of $p.$ Following the proofs of Theorem \ref{thm: hofer's energy bd by symplectic area-1}
and Theorem \ref{thm:(Gromov's-Monotonicity-Theorem} carefully, we
can see that the constant $C_{2}$ can be chosen to be close to $1$
and $C_{4}$ can be bounded using $K_{0}.$ Therefore, we can make
$\frac{4C_{2}(2\pi)}{C_{4}}$ independent of $p.$ \end{proof}

\end{document}